\documentclass{amsart}

\usepackage{graphicx}

\newtheorem{theorem}{Theorem}[section]

\newtheorem{problem}[theorem]{Problem}

\theoremstyle{definition}

\theoremstyle{remark}

\numberwithin{equation}{section}

\newcommand{\Sdim}{\mbox{\rm S-dim}}
\newcommand{\SDim}{\mbox{\rm S-Dim}}
\newcommand{\U}{\mathcal U}
\newcommand{\C}{\mathcal C}
\newcommand{\Lip}{\mathrm{Lip}}
\newcommand{\e}{\varepsilon}
\newcommand{\phii}{\varphi}
\newcommand{\IN}{\mathbb N}
\newcommand{\Z}{\mathbb Z}

\newcommand{\IR}{\mathbb R}
\newcommand{\diam}{\mathrm{diam}}

\begin{document}

\title{A 1-dimensional Peano continuum\\ which is not an IFS attractor}
\author{Taras Banakh}
\address{Instytut Matematyki, Jan Kochanowski University, Kielce, Poland and\newline Faculty of Mechanics and Mathematics, Ivan Franko National University of Lviv, Ukraine}
\email{t.o.banakh@gmail.com}

\author{Magdalena Nowak}
\address{Instytut Matematyki, Jan Kochanowski University, Kielce, Poland and\newline Instytut Matematyki, Jagiellonian University, Krakow, Poland }
\email{magdalena.nowak805@gmail.com}
\thanks{The second author was supported in part by PHD fellowships important for regional development}


\commby{}

\begin{abstract}
Answering an old question of M.Hata, we construct an example of a 1-dimensional Peano continuum which is~not homeomorphic to an attractor of IFS.
\end{abstract}

\maketitle

\section{Introduction}

A compact metric space $X$ is called an {\em IFS-attractor} if $X=\bigcup_{i=1}^nf_i(X)$ for some contracting self-maps $f_1,\dots,f_n:X\to X$. In this case the family $\{f_1,\dots,f_n\}$ is~called an {\em iterated function system} (briefly, an IFS), see \cite{B}. We recall that a~map $f:X\to X$ is {\em contracting} if its Lipschitz constant
\begin{equation*}
\Lip(f)=\sup_{x\ne y}\frac{d(f_i(x),f_i(y))}{d(x,y)}
\end{equation*}
is less than 1.

Attractors of IFS appear naturally in the Theory of Fractals, see \cite{B}, \cite{Ed}. 
Topological properties of IFS-attractors were studied by M.Hata in \cite{Hata}. In particular, he  observed that each connected IFS-attractor $X$ is locally connected. The reason is that $X$ has property S. We recall \cite[8.2]{Nad} that a metric space $X$ has {\em property S} if for every $\e>0$ the space $X$ can be covered by finite number of connected subsets of diameter $<\e$. It is well-known \cite[8.4]{Nad} that a connected compact metric space $X$ is locally connected if and only if it has property $S$ if and only if $X$ is a {\em Peano continuum} (which means that $X$ is the continuous image of the interval $[0,1]$). Therefore, a compact space $X$ is not homeomorphic to an~IFS-attractor whenever $X$ is connected but not locally connected. Now it is natural to ask if~there is a Peano continuum homeomorphic to no IFS-attractor. An~easy answer is ``Yes'' as every IFS-attractor has finite topological dimension, see \cite{Ed}. Consequently, no infinite-dimensional compact topological space is  homeomorphic to an IFS-attractor. In such a way we arrive to the following question posed by M.~Hata in \cite{Hata}.

\begin{problem}\label{hata} Is each finite-dimensional Peano continuum homeomorphic to an IFS-attractor?
\end{problem}

In this paper we shall give a negative answer to this question. Our counterexample is a rim-finite plane Peano continuum. A topological space $X$ is called {\em rim-finite} if it has a base of the topology consisting of open sets with finite boundaries. It follows that each compact rim-finite space $X$ has dimension $\dim(X)\le 1$.

\begin{theorem}\label{th1} There is a rim-finite plane Peano continuum homeomorphic to no IFS-attractor.
\end{theorem}

It should be mentioned that an example of a Peano continuum $K\subset\IR^2$, which is not {\em isometric} to an IFS-attractor was constructed by M.Kwieci\'nski in~\cite{Kw}. However the continuum of Kwieci\'nski is {\em homeomorphic} to an IFS-attractor, so it does~not give an answer to Problem~\ref{hata}.

\section{S-dimension of IFS-attractors}\label{s1}

In order to prove Theorem~\ref{th1} we shall observe that each connected IFS-attractor has finite $S$-dimension. This dimension was introduced and studied in \cite{BT}.

The $S$-dimension $\SDim(X)$ is defined for each metric space $X$ with property $S$. For each $\e>0$ denote  by $S_\e(X)$ the smallest number of connected subsets of~diameter $<\e$ that cover the space $X$ and let $$\SDim(X)=\operatornamewithlimits{\overline{\lim}}\limits_{\e\to+0}-\frac{\ln S_\e(X)}{\ln\e}.$$ For each Peano continuum $X$ we can also consider a topological invariant
\begin{equation*}
\Sdim(X)=\inf\{\SDim(X,d):\mbox{$d$ is a metric generating the topology of $X$}\}.
\end{equation*}
By \cite[5.1]{BT}, $\Sdim(X)\ge\dim(X)$, where $\dim(X)$ stands for the covering topological dimension of $X$.

\begin{theorem}\label{Sdim} Assume that a connected compact metric space $X$ is an attractor of~an~IFS $f_1,f_2,\dots,f_n:X\to X$ with contracting constant $\lambda=\max_{i\le n}\Lip(f_i)<1$. Then $X$ has finite $S$-dimensions \begin{equation*}
\Sdim(X)\le \SDim(X)\le -\frac{\ln(n)}{\ln(\lambda)}.
\end{equation*}
\end{theorem}

\begin{proof} The inequality $\Sdim(X)\le\SDim(X)$ follows from the definition of the $S$-dimension $\Sdim(X)$. The inequality $\SDim(X)\le-\frac{\ln(n)}{\ln(\lambda)}$ will follow as soon as for~every $\delta>0$ we find $\e_0>0$ such that for every $\e\in(0,\e_0]$ we get
\begin{equation*}
-\frac{\ln S_\e(X)}{\ln\e}<-\frac{\ln(n)}{\ln(\lambda)}+\delta.
\end{equation*}

Let $D=\diam(X)$ be the diameter of the metric space $X$. Since
\begin{equation*}
\lim_{k\to\infty}\frac{\ln(n^k)}{\ln(\lambda^{k-1}D)}=\lim_{k\to\infty}\frac{k\ln(n)}{(k-1)\ln(\lambda)+\ln D}=\frac{\ln(n)}{\ln(\lambda)},
\end{equation*}
there is $k_0\in\IN$ such that for each $k\ge k_0$ we get
\begin{equation*}
-\frac{\ln(n^k)}{\ln(\lambda^{k-1}D)}<-\frac{\ln(n)}{\ln(\lambda)}+\delta.
\end{equation*}

We claim that the number $\e_0=\lambda^{k_0-1}D$ has the required property. Indeed, given any $\e\in(0,\e_0]$ we can find $k\ge k_0$ with $\lambda^{k}D<\e\le\lambda^{k-1}D$ and observe that
\begin{equation*}
\C_k=\big\{f_{i_1}\circ\dots\circ f_{i_k}(X):i_1,\dots,i_k\in\{1,\dots,n\}\big\}
\end{equation*}
is a cover of $X$ by compact connected subsets having diameter $\le\lambda^kD< \e$. Then
$S_\e(X)\le |\C_k|\le n^k$ and
\begin{equation*}
-\frac{\ln(S_\e(X))}{\ln(\e)}\le -\frac{\ln(n^k)}{\ln(\lambda^{k-1}D)}<-\frac{\ln(n)}{\ln(\lambda)}+\delta.
\end{equation*}
\end{proof}

In the next section we shall construct an example of a rim-finite plane Peano continuum $M$ with~infinite $S$-dimension $\Sdim(M)$. Theorem~\ref{Sdim} implies that the~space $M$ is not homeomorphic to an IFS-attractor and this proves Theorem~\ref{th1}.

\section{The space $M$}\label{s2}

Our space $M$ is a partial case of the spaces constructed in \cite{BT} and called "shark teeth".
Consider the piecewise linear periodic function
\begin{equation*}
\phii(t)=
 \begin{cases}
t-n & \text{if }t\in[n,n+\frac{1}{2}] \text{ for some }n\in\Z,  \\
n-t & \text{if }t\in[n-\frac{1}{2},n] \text{ for some }n\in\Z,
\end{cases}
\end{equation*}
whose graph looks as follows:

\begin{picture}(300,80)(-30,-10)
\put(-20,0){\vector(1,0){300}}
\put(275,-10){$t$}
\put(0,-10){\vector(0,1){70}}
\put(-20,20){\line(1,-1){20}}
\put(0,0){\line(1,1){40}}
\put(40,40){\line(1,-1){40}}
\put(80,0){\line(1,1){40}}
\put(120,40){\line(1,-1){40}}
\put(160,0){\line(1,1){40}}
\put(200,40){\line(1,-1){40}}
\put(240,0){\line(1,1){20}}
\end{picture}

For every $n\in\IN$ consider the function
\begin{equation*}
\varphi_n(t)=2^{-n}\varphi(2^nt),
\end{equation*}
which is a homothetic copy of the function $\varphi(t)$.

Consider the non-decreasing sequence
\begin{equation*}
n_k=\lfloor \log_2 \log_2 (k+1)\rfloor,\quad k\in\IN,
\end{equation*}
where $\lfloor x\rfloor$ is the integer part of $x$. Our example is the continuum
\begin{equation*}
M=[0,1]\times\{0\}\cup
\bigcup_{k=1}^\infty\big\{\big(t,\tfrac1k\varphi_{n_k}(t)\big):t\in[0,1]\big\}
\end{equation*}
in the plane $\IR^2$, which looks as follows:

\begin{figure}[h]
\begin{center}	
\includegraphics[scale=0.9]{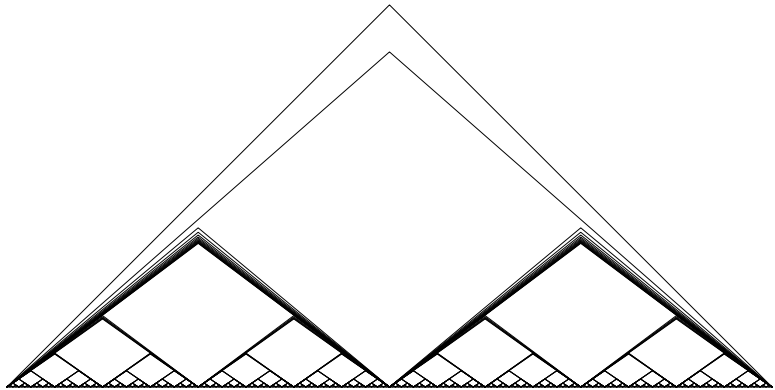}
	\caption{The continuum $M$}
\end{center}
\end{figure}

The following theorem describes some properties of the continuum $M$ and implies Theorem~\ref{th1} stated in the introduction.

\begin{theorem}\label{tM} The space $M$ has the following properties:
\begin{enumerate}
\item $M$ is a rim-finite plane Peano continuum;
\item $\dim(M)=1$ and $\Sdim(X)=\infty$;
\item $M$ is not homeomorphic to an IFS attractor.
\end{enumerate}
\end{theorem}

\begin{proof} It is easy to see that $X$ is a rim-finite plane Peano continuum.
The rim-finiteness of $M$ implies that $\dim(M)=1$. To show that $\Sdim(M)=\infty$, consider the number sequence $\vec m=(2^{n_k})_{k=1}^\infty$ and observe that the space $M$ is homeomorphic to the ``shark teeth'' space $W_{\vec m}$ considered in \cite{BT}. Taking into account that
\begin{equation*}
\lim_{k\to\infty}\frac{2^{n_k}}{k^\alpha}=0\mbox{ \ \ for any $\alpha>0$}
\end{equation*}
 and applying Theorem~7.3(6) of \cite{BT}, we conclude that $\Sdim M=\Sdim W_{\vec m}=\infty$.
By Theorem~\ref{Sdim}, the space $M$ is not homeomorphic to an IFS-attractor.
\end{proof}

\section{Some Open Questions}

We shall say that a compact topological space $X$ is a {\em topological IFS-attractor} if $X=\bigcup_{i=1}^nf_i(X)$ for some continuous maps $f_1,\dots,f_n:X\to X$ such for any open cover $\U$ of $X$ there is $m\in\IN$ such that for any functions $g_1,\dots,g_m\in\{f_1,\dots,f_n\}$ the set $g_1\circ\dots\circ g_m(X)$ lies in some set $U\in\U$. It is easy to see that each IFS-attractor is a topological IFS-attractor and each connected topological IFS-attractor is metrizable and locally connected.

\begin{problem} Is each (finite-dimensional) Peano continuum a topological IFS-attrac\-tor? In particular, is the space $M$ constructed in Theorem~\ref{tM} a topological IFS-attractor?
\end{problem}

\section{Acknowledgment}

The authors express their sincere thanks to Wies{\l}aw Kubi\'s for interesting discussions and valuable comments.

\bibliographystyle{amsplain}

\end{document}